\documentclass{amsart}
\usepackage{mathrsfs}
\usepackage{mathrsfs}
\usepackage{amsfonts}
\usepackage{cases}
\usepackage{latexsym}
\usepackage{amsmath}
\usepackage[all]{xy}
\usepackage{stmaryrd}
\usepackage{amsfonts}
\usepackage{amsmath,amssymb,amscd,bbm,amsthm,mathrsfs,dsfont}

\usepackage{fancyhdr}
\usepackage{amsxtra,ifthen}
\usepackage{verbatim}

\numberwithin{equation}{section}

\theoremstyle{plain}
\newtheorem{theorem}{Theorem}[section]
\newtheorem{lemma}[theorem]{Lemma}
\newtheorem{proposition}[theorem]{Proposition}
\newtheorem{corollary}[theorem]{Corollary}

\theoremstyle{definition}

\newtheorem{example}[theorem]{Example}

\newtheorem{remark}[theorem]{Remark}

\newtheorem{question}[theorem]{Question}

\begin{document}

\title[Jordan Derivations and Lie derivations on Path algebras]
{Jordan Derivations and Lie derivations on Path Algebras}

\author{Yanbo Li and Feng Wei}

\address{Li: Department of Information and Computing Sciences, Northeastern
University at Qinhuangdao, Qinhuangdao, 066004, P.R. China}

\email{liyanbo707@163.com}

\address{Wei: School of Mathematics, Beijing Institute of
Technology, Beijing, 100081, P. R. China}

\email{daoshuo@hotmail.com}

\begin{abstract}
Without the faithful assumption, we prove that every Jordan
derivation on a class of path algebras of quivers without oriented
cycles is a derivation and that every Lie derivation on such kinds
of algebras is of the standard form.
\end{abstract}

\subjclass[2000]{15A78, 16W25}

\keywords{Path algebra, Jordan derivation, Lie derivation}

\thanks{The work of the first author is supported by Fundamental Research
Funds for the Central Universities (N110423007). The work of the
second author is partially supported by the National Natural Science
Foundation of China (Grant No. 10871023).}

\maketitle
\section{Introduction}

Let $\mathcal{R}$ be a commutative ring with identity, $\mathcal{A}$
be a unital algebra over $\mathcal{R}$ and $\mathcal{Z(A)}$ be the
center of $\mathcal{A}$. We set $a\circ b=ab+ba$ and $[a, b]=ab-ba$
for all $a,b\in \mathcal{A}$. Recall that an $\mathcal{R}$-linear
mapping $\Theta$ from $\mathcal{A}$ into itself is called a
\textit{derivation} if
$$
\Theta(ab)=\Theta(a)b+a\Theta(b)
$$
for all $a, b\in \mathcal{A}$, a \textit{Jordan derivation} if
$$
\Theta(a\circ b)=\Theta(a)\circ b+a \circ \Theta(b)
$$
for all $a\in \mathcal{A}$, and a \textit{Lie derivation} if
$$
\Theta([a, b])=[\Theta(a), b]+[a, \Theta(b)]
$$
for all $a, b\in \mathcal{A}$. Moreover, in the $2$-torsion free
case the definition of a Jordan derivation is equivalent to $
\Theta(a^2)=\Theta(a)a+a\Theta(a)$ for all $a\in \mathcal{A}$. We
say a Lie derivation $\Theta$ is \textit{standard} if it can be
expressed as
$$
\Theta=D+\Phi, \eqno(\spadesuit)
$$
where $D$ is an ordinary derivation of $\mathcal{A}$ and $\Phi$ is a
linear mapping from $\mathcal{A}$ into the center $\mathcal{Z(A)}$
of $\mathcal{A}$.

Jordan derivations and Lie derivations of associative algebras play
significant roles in various mathematical areas, in particular in
matrix theory, in ring theory and in the theory of operator
algebras. There are two common problems in this context:

\begin{enumerate}
\item[(a)] whether Jordan derivations are derivations;

\item[(b)] whether Lie derivations are of the
standard form $(\spadesuit)$.
\end{enumerate}
Let $\mathcal{A,B}$ be two unital algebras over the commutative ring
$\mathcal{R}$, $\mathcal{M}$ be a faithful $(\mathcal{A,
B})$-bimodule and $\mathcal{N}$ be a $(\mathcal{B, A})$-bimodule. We
denote the full matrix algebra of all $n\times n$ matrices over
$\mathcal{R}$, the triangular algebra consisting of $\mathcal{A, B}$
and $\mathcal{M}$, and the generalized matrix algebra consisting of
$\mathcal{A, B, M}$ and $\mathcal{N}$ by
$$
M_{n\times n}(\mathcal{R}), \hspace{6pt} \mathcal{T_R}=\left[
\begin{array}
[c]{cc}%
\mathcal{A} & \mathcal{M}\\
0 & \mathcal{B}\\
\end{array}
\right] \hspace{6pt} {\rm and} \hspace{6pt} \mathcal{G_R}=\left[
\begin{array}
[c]{cc}%
\mathcal{A} & \mathcal{M}\\
\mathcal{N} & \mathcal{B}\\
\end{array}
\right]
$$
respectively. In \cite{JacobsonRickart} Jacobson and Rickart proved
that every Jordan derivation on $M_{n\times n}(\mathcal{R})$ is a
derivation. Zhang and Yu \cite{ZhangYu} obtained that every Jordan
derivation on $\mathcal{T_R}$ is a derivation. Xiao and Wei
\cite{XiaoWei1} extended this result to the higher case and obtained
that any Jordan higher derivation on a triangular algebra is a
higher derivation. Alaminos et al \cite{AlaminosBresarVillena}
showed that every Lie derivation on the full matrix algebra
$M_{n\times n}(\mathbb{F})$ of all $n\times n$ matrices over a field
$\mathbb{F}$ of characteristic zero has the standard form
$(\spadesuit)$. Cheung \cite{Cheung} considered Lie derivations of
triangular algebras and gave a sufficient and necessary condition
which enables every Lie derivations of $\mathcal{T_R}$ to be
standard $(\spadesuit)$. Yu and Zhang extended Cheung's work to the
nonlinear case \cite{YuZhang}. Benkovic investigated the structure
of Jordan derivations and Lie derivations from $\mathcal{T_R}$ into
its bimodule in \cite{Benkovic1} and \cite{Benkovic3}. More
recently, the current authors \cite{LiWei1, LiWei2} described the
structures of Jordan and Lie derivations on the generalized matrix
algebra $\mathcal{G_R}$. It was shown that under certain conditions,
every Jordan derivation on $\mathcal{G_R}$ can be decomposed as the
sum of a derivation and an anti-derivation. We also proved that
under some mild assumptions, every Lie derivation on $\mathcal{G_R}$
has the standard form $(\spadesuit)$. In \cite{Benkovic4}
Benkovi\v{c} studied generalized Jordan derivations and generalized
Lie derivations of the triangular algebra $\mathcal{T_R}$ and
investigated whether they also have the nice properties which are
similar to those of Jordan derivations and Lie derivations.

We need to point out that most of existing works related to Jordan
derivations and Lie derivations of matrix algebras heavily depend on
the faithful condition. For instance, the $(\mathcal{A,
B})$-bimodule $\mathcal{M}$ in $\mathcal{T_R}$ and $\mathcal{G_R}$
is always assumed to be faithful. There is a commonly consensus in
the study of related topics. It seems that the assumption concerning
the faithfulness is a very natural condition and that without it,
one hardly get any useful results. To the best of our knowledge
there are no any articles treating Jordan derivations and Lie
derivations of matrix algebras without the faithful assumption
except for \cite{Cheung}.  When studying some harder problems, it
sometimes occurs that even the assumption concerning faithfulness is
too weak and then the problem has to be approached via some stronger
assumptions, for example loyalty ($a\mathcal{M}b=0$ implies that
$a=0$ or $b=0$), which was already used when studying the Lie
isomorphisms on triangular algebras. Therefore we propose a
challenging question:
\begin{question}
Without the faithful assumption, what can we say about the Jordan
derivations and Lie derivations of matrix algebras?
\end{question}

Path algebras of quivers come up naturally in the study of tensor
algebras of bimodules over semisimple algebras. It is well known
that any finite dimensional basic $\mathcal{K}$-algebra is given by
a quiver with relations when $\mathcal{K}$ is  an algebraically
closed field. In \cite{GuoLi}, Guo and Li studied the Lie algebra of
differential operators on a path algebra $\mathcal {K}\Gamma$ and
related this Lie algebra to the algebraic and combinatorial
properties of the path algebra $\mathcal {K}\Gamma$. The main
purpose of this paper is to study Jordan derivations and Lie
derivations of a class of path algebras of quivers without oriented
cycles, which can be viewed as one-point extensions. The
distinguished feature of our work is that the faithful assumption is
removed. We prove that every Jordan derivation on a class of path
algebras of quivers without oriented cycles is a derivation and that
every Lie derivation on such kinds of algebras is of the standard
form $(\spadesuit)$.

\section{Path algebras and triangular algebras}
Let us give a quick review of path algebras of quivers and
triangular algebras. For more details, we refer the reader to
\cite{AuslanderReitenSmalo}.

\subsection{Path algebras}
A quiver $\Gamma$ is an oriented graph. Let us denote the set of
vertices by $\Gamma_0$ and denote the set of arrows between vertices
by $\Gamma_1$. Throughout this paper, we always assume that
$\Gamma_0$ and $\Gamma_1$ are both finite sets. In this case, we say
that the quiver $\Gamma=(\Gamma_0, \Gamma_1)$ is \textit{finite}. If
$\alpha$ is an arrow from the vertex $i$ to the vertex $j$, then we
write $s(\alpha)=i$ and $e(\alpha)=j$. A vertex $i$ is called a
\textit{sink} if there is no arrow $\alpha$ such that $s(\alpha)=i$
and is called a \textit{source} if there is no arrow $\alpha$ such
that $e(\alpha)=i$. A \textit{nontrivial path} in $\Gamma$ is an
ordered sequence of arrows $p=\alpha_n\cdots\alpha_1$ such that
$e(\alpha_m)=s(\alpha_{m+1})$ for each $1\leq m<n$. Define
$s(p)=s(\alpha_1)$ and $e(p)=e(\alpha_n)$. A \textit{trivial path}
is the symbol $e_i$ for each $i\in \Gamma_0$. In this case, we set
$s(e_i)=e(e_i)=i$. A nontrivial path $p$ is called an
\textit{oriented cycle} if $s(p)=e(p)$. Denote the set of all paths
by $\mathscr{P}$.

Let $\mathcal{K}$ be a field and $\Gamma$ be a quiver. Then the path
algebra ${\mathcal K}\Gamma$ is the ${\mathcal K}$-algebra generated
by the paths in $\Gamma$ and the product of two paths
$x=\alpha_n\cdots\alpha_1$ and $y=\beta_t\cdots\beta_1$ is defined
by
$$
xy=\left\{
\begin{array}{ll}
\alpha_n\cdots\alpha_1\beta_t\cdots\beta_1, & \mbox{$e(y)=s(x)$}\\
0, & \mbox{otherwise}
\end{array}
\right.
$$
Clearly, ${\mathcal K}\Gamma$ is an associative algebra with the
identity $1=\sum_{i\in \Gamma_0}e_i$, where $e_i(i\in \Gamma_0)$ are
pairwise orthogonal primitive idempotents of ${\mathcal K}\Gamma$.

A relation $\sigma$ on a quiver $\Gamma$ over a field ${\mathcal K}$
is a ${\mathcal K}$-linear combination of paths
$\sigma=\sum_{i=1}^nk_ip_i$, where $k_i\in {\mathcal K}$ and
$e(p_1)=\cdots=e(p_n)$, $s(p_1)=\cdots=s(p_n)$. Moreover, the number
of arrows in each path is assumed to be at least 2. Let $\rho$ be a
set of relations on $\Gamma$ over ${\mathcal K}$. The pair $(\Gamma,
\rho)$ is called a \textit{quiver} with relations over ${\mathcal
K}$. Denote by $<\rho>$ the ideal of ${\mathcal K}\Gamma$ generated
by the set of relations $\rho$. The ${\mathcal K}$-algebra
${\mathcal K}(\Gamma, \rho)={\mathcal K}\Gamma/<\rho>$ is always
associated with $(\Gamma, \rho)$. For arbitrary element $x\in
{\mathcal K}\Gamma$, write by $\overline x$ the corresponding
element in ${\mathcal K}(\Gamma, \rho)$. It is well known that every
basic finite dimensional algebra over an algebraically closed field
${\mathcal K}$ is isomorphic to some ${\mathcal K}(\Gamma, \rho)$.

\begin{example}
Let ${\mathcal K}$ be a field and $\Gamma$ be the following quiver
$$\xymatrix@C=13mm{
  \bullet
  \ar@<0pt>[r]^(0){1}^{\alpha_1}  &
  \bullet
  \ar@<0pt>[r]^(0.3){\alpha_2}^(0){2}^(0.7){3}&\bullet\,\,\,
  \cdots\cdots\,\,\,
  \bullet
  \ar@<0pt>[r]^(0.6){\alpha_{n-1}}^(0.3){n-1}^(0.99){n}&
  \bullet}
  $$
  Then ${\mathcal K}\Gamma$ is isomorphic to the upper triangular matrix
  algebra $T_n(K)$
\end{example}

\begin{example}
Let ${\mathcal K}$ be a field and $\Gamma$ be the following quiver
$$
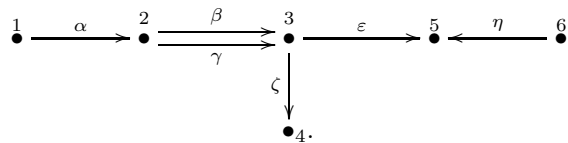
\xymatrix@C=13mm{
  \bullet
  \ar@<0pt>[r]^(0){1}^{\alpha}  &
  \bullet
  \ar@<2.5pt>[r]^(0.5){\beta}^(0){2}^(1){3}
  \ar@<-2.5pt>[r]_(0.5){\gamma}&
  \bullet
  \ar[d]_(0.5){\zeta}
  \ar@<0pt>[r]^{\varepsilon}  &
  \bullet &
  \bullet
  \ar@<0pt>@[r][l]_(0){6}_(1){5}_{\eta} &\\
  & & \,\,\,\, \bullet_{4} .}
  $$
Then a basis of the path algebra ${\mathcal K}\Gamma$ is
$$
\{e_1,
e_2, e_3, e_4, e_5, e_6, \alpha, \beta, \gamma, \zeta, \varepsilon,
\eta, \beta\alpha, \gamma\alpha, \zeta\beta\alpha,
\zeta\gamma\alpha, \varepsilon\beta\alpha, \varepsilon\gamma\alpha,
\zeta\beta, \zeta\gamma, \varepsilon\beta, \varepsilon\gamma\}.
$$
If $\rho=\{\varepsilon\beta\}$, then a basis of ${\mathcal
K}(\Gamma, \rho)$ is
$$
\{\overline e_1,
\overline e_2, \overline e_3, \overline e_4, \overline e_5,
\overline e_6, \overline\alpha, \overline\beta, \overline\gamma,
\overline\zeta, \overline\varepsilon, \overline\eta,
\overline{\beta\alpha}, \overline\gamma\overline\alpha,
\overline{\zeta\beta\alpha}, \overline{\zeta\gamma\alpha},
\overline\varepsilon\overline\gamma\overline\alpha,
\overline\zeta\overline\beta, \overline{\zeta\gamma},
\overline\varepsilon\overline\gamma\}.$$
\end{example}

\subsection{Triangular algebras and one-point extension}
Let us begin with the definition of triangular algebras over a field
${\mathcal K}$. Let ${\mathcal K}$ be a field and $\mathcal{A}$ and
$B$ two ${\mathcal K}$-algebras. Let $_{\mathcal A}{\mathcal
M}_{\mathcal B}$ be an $({\mathcal A, B})$-bimodule. Note that
$\mathcal{M}$ need not to be faithful neither as left ${\mathcal
A}$-module nor as right ${\mathcal B}$-module here. Then one can
define
$$
\mathcal{T_K}=\left[
\begin{array}
[c]{cc}%
{\mathcal A} & {\mathcal M}\\
0 & {\mathcal B}\\
\end{array}
\right]=\left\{ \hspace{2pt} \left[
\begin{array}
[c]{cc}%
a & m\\
0 & b\\
\end{array}
\right] \hspace{2pt} \vline \hspace{2pt} a\in {\mathcal A}, b\in
{\mathcal B}, m\in {\mathcal M} \hspace{2pt} \right\}
$$
to be an associative algebra under matrix-like addition and
matrix-like multiplication. The center of $\mathcal{T_K}$ is
$$
\mathcal{Z(T_K)}=\left\{ \left[
\begin{array}
[c]{cc}%
a & 0\\
0 & b
\end{array}
\right] \vline \hspace{3pt}a\in \mathcal {Z}(\mathcal {A}), b\in
\mathcal {Z}(\mathcal {B}), am=mb, \ \forall\ m\in {\mathcal M}
\right\}.
$$

In particular, if $\mathcal{B}$ is equal to the field ${\mathcal
K}$, then the triangular algebra
$\mathcal{T_K}=\left[\smallmatrix {\mathcal A} & {\mathcal M}\\
0 & {\mathcal K} \endsmallmatrix \right]$ is called a
\textit{one-point extension} of ${\mathcal A}$ by the bimodule
$_{\mathcal A}{\mathcal M}_{\mathcal B}$. This terminology comes up
in connection with path algebras. For convenience, we explain the
reason here. Let $\Lambda={\mathcal K}(\Gamma, \rho)$ be a finite
dimensional path algebra of the quiver $(\Gamma, \rho)$ with
relations. Let $i$ be a source in $\Gamma$ and ${\overline e_i}$ the
corresponding idempotent in $\Lambda$. Note that $\Gamma$ is a
quiver without oriented cycles. Clearly, there are no nontrivial
paths ending in $i$. This implies that ${\overline e_i}\Lambda
{\overline e_i}\simeq {\mathcal K}$ and ${\overline e_i}\Lambda
(1-{\overline e_i})=0$. Therefore
$$
\Lambda\simeq \left[
\begin{array}
[c]{cc}%
(1-{\overline e_i})\Lambda(1-{\overline e_i}) & (1-{\overline e_i})\Lambda {\overline e_i}\\
0 & {\mathcal K}\\
\end{array}
\right].
$$
Let us denote by $(\Gamma', \rho')$ the quiver obtained by removing
the vertex $i$ and the relations starting at $i$ and write
$\Lambda'= {\mathcal K}(\Gamma', \rho')$. Then $(1-{\overline
e_i})\Lambda(1-{\overline e_i})\simeq\Lambda'$. Thus $\Lambda$ can
be constructed from $\Lambda'$ by adding one point $i$, together
with arrows and relations which start at $i$.

It is worth noting that $(1-{\overline e_i})\Lambda {\overline e_i}$
is not faithful as a left $\Lambda'$-module in general. We
illustrate an example here.

\begin{example}\label{xxsec2.2}
Let ${\mathcal K}$ be a field and $\Lambda={\mathcal K}\Gamma$ a
path algebra, where $\Gamma$ is the following quiver
$$
\xymatrix@C=13mm{
  \bullet
  \ar@<0pt>[r]_(0){1}^{\alpha}  &
  \bullet &
  \bullet
  \ar@<0pt>@[r][l]_(0.5){\beta}^(0){3}^(1){2}
  \ar@<2.5pt>[r]^(0.5){\gamma}_(1.1){4}
  \ar@<-2.5pt>[r]_(0.5){\delta} &
  \bullet }$$

Obviously, vertex 1 is a source in $\Gamma$. Then
$$\Lambda\simeq \left[
\begin{array}
[c]{cc}%
(1-e_1)\Lambda(1-e_1) & (1-e_1)\Lambda e_1\\
0 & {\mathcal K}\\
\end{array}
\right].
$$
It is easy to check that $(1-e_1)\Lambda e_1$ is not faithful as a
left $(1-e_1)\Lambda(1-e_1)$ module.

On the other hand, let us take $e=e_1+e_2$. It is easy to check that
$(1-e)\Lambda e=0$. Then the algebra $\mathcal {K}\Gamma$ can also
be viewed as a triangular algebra as follows:
$$\Lambda\simeq \left[
\begin{array}
[c]{cc}%
e\Lambda e & e\Lambda (1-e)\\
0 & (1-e)\Lambda(1-e)\\
\end{array}
\right].
$$
Clearly, $e\Lambda (1-e)$ is not faithful as left $e\Lambda
e$-module and as right $(1-e)\Lambda(1-e)$-module.
\end{example}

\section{Jordan Derivations on Path Algebras}\label{xxsec3}

Let us first recall some indispensable descriptions concerning derivations and Jordan
derivations of the triangular algebra $\mathcal{T_K}$.

\begin{lemma}\cite[Lemma 5]{Cheung}\label{xxsec3.1}
An $\mathcal{K}$-linear mapping $\Theta$ from $\mathcal{T_K}=\left[\smallmatrix {\mathcal A} & {\mathcal M}\\
0 & {\mathcal B} \endsmallmatrix \right]$ into itself is a derivation if and
only if it has the form
$$
\begin{aligned}
& \Theta\left(\left[
\begin{array}
[c]{cc}%
a & m\\
0 & b\\
\end{array}
\right]\right) = \left[
\begin{array}
[c]{cc}%
\delta_1(a) & am_0-m_0b+\tau_2(m)\\
0 & \mu_4(b)\\
\end{array}
\right] , \quad \forall \left[
\begin{array}
[c]{cc}%
a & m\\
0 & b\\
\end{array}
\right]\in \mathcal{T_K},
\end{aligned}
$$
where $m_0\in {\mathcal M}$ and
$$
\begin{aligned} \delta_1:& {\mathcal A} \longrightarrow {\mathcal A}, &
 \tau_2: & {\mathcal M}\longrightarrow {\mathcal M}, &
\mu_4: & {\mathcal B}\longrightarrow {\mathcal B}
\end{aligned}
$$
are all ${\mathcal K}$-linear mappings satisfying the following conditions
\begin{enumerate}
\item[(1)] $\delta_1$ is a derivation of ${\mathcal A}$ and $\mu_4$
is a derivation of ${\mathcal B}$.

\item[(2)] $\tau_2(am)=a\tau_{2}(m)+\delta_1(a)m$ and
$\tau_2(mb)=\tau_2(m)b+m\mu_4(b).$
\end{enumerate}
\end{lemma}

\begin{lemma}\cite[Lemma 3.2]{AB}\label{xxsec3.2}
Let ${\mathcal K}$ be a field with ${\rm Char} {\mathcal K}\neq 2$.
An $\mathcal{K}$-linear
mapping $\Theta$ from $\mathcal{T_K}=\left[\smallmatrix {\mathcal A} & {\mathcal M}\\
0 & {\mathcal B} \endsmallmatrix \right]$ into itself is a Jordan
derivation if and only if it has the form
$$
\begin{aligned}
& \Theta\left(\left[
\begin{array}
[c]{cc}%
a & m\\
0 & b\\
\end{array}
\right]\right) = \left[
\begin{array}
[c]{cc}%
\delta_1(a) & am_0-m_0b+\tau_2(m)\\
0 & \mu_4(b)\\
\end{array}
\right] ,\quad \forall \left[
\begin{array}
[c]{cc}%
a & m\\
0 & b\\
\end{array}
\right]\in \mathcal{T_K},
\end{aligned}
$$
where $m_0\in {\mathcal M}$ and
$$
\begin{aligned} \delta_1:& {\mathcal A} \longrightarrow {\mathcal A},  &  \tau_2:
& {\mathcal M}\longrightarrow {\mathcal M}, &  \mu_4: & {\mathcal B}\longrightarrow {\mathcal B}
\end{aligned}
$$
are all ${\mathcal K}$-linear mappings satisfying conditions
\begin{enumerate}
\item [(1)] $\delta_1$ is a Jordan derivation of ${\mathcal A}$ and
$\mu_4$ is a Jordan derivation of ${\mathcal B}$.

\item[(2)] $\tau_2(am)=a\tau_2(m)+\delta_1(a)m$ and
$\tau_2(mb)=\tau_2(m)b+m\mu_4(b).$
\end{enumerate}
\end{lemma}

It should be remarked that if  $\Gamma$ is a quiver without oriented
cycles, then the path algebra ${\mathcal K}(\Gamma, \rho)$ can be
viewed as a one-point extension algebra. We now give the form of
Jordan derivations in this background.

\begin{lemma}\label{xxsec3.3}
Let $\mathcal{T_K}=\left[\smallmatrix {\mathcal A} & {\mathcal M}\\
0 & {\mathcal K} \endsmallmatrix \right]$ be a one-point extension
of $\mathcal {A}$ by the bimodule $_\mathcal {A}\mathcal
{M}_\mathcal {K}$ with ${\mathcal M}\neq 0$ and $\Theta$ be a Jordan
derivation of $\mathcal{T_K}$. Then $\Theta$ has the form
$$
\begin{aligned}
& \Theta\left(\left[
\begin{array}
[c]{cc}%
a & m\\
0 & k\\
\end{array}
\right]\right) = \left[
\begin{array}
[c]{cc}%
\delta_1(a) & am_0-m_0k+\tau_2(m)\\
0 & 0\\
\end{array}
\right] ,\quad \forall \left[
\begin{array}
[c]{cc}%
a & m\\
0 & k\\
\end{array}
\right]\in \mathcal{T_K},
\end{aligned}
$$
where $m_0\in {\mathcal M}$ and both $\delta_1: {\mathcal
A}\longrightarrow {\mathcal A}$ and $\tau_2: {\mathcal
M}\longrightarrow {\mathcal M}$ are ${\mathcal K}$-linear mappings
satisfying conditions
\begin{enumerate}
\item [(1)] $\delta_1$ is a Jordan derivation of ${\mathcal A}$.

\item [(2)] $\tau_2(am)=a\tau_2(m)+\delta_1(a)m.$
\end{enumerate}
\end{lemma}

\begin{proof}
Clearly, by Lemma  \ref{xxsec3.2}, we only need to show $\mu_4(k)=0$
for all $k\in {\mathcal K}$. In fact, it follows from $\tau_2$ being
${\mathcal K}$-linear that $\tau_2(mk)=\tau_2(m)k$ for all $m\in
{\mathcal M}$ and $k\in {\mathcal K}$. Then the condition (2) of
Lemma \ref{xxsec3.2} implies that $m\mu_4(k)=0$ for all $k\in
{\mathcal K}$ and $m\in {\mathcal M}$. Note that ${\mathcal M}$ is a
nonzero ${\mathcal K}$-vector space. This forces $\mu_4(k)=0$ for
all $k\in {\mathcal K}$.
\end{proof}

Now we are in a position to state the main result of this section.

\begin{theorem}\label{xxsec3.4}
Let ${\mathcal K}$ be a field with ${\rm Char} {\mathcal K} \neq 2$
and $\Lambda={\mathcal K}(\Gamma, \rho)$ a finite dimensional path
algebra over ${\mathcal K}$ of the quiver $(\Gamma, \rho)$ with
relations. Then every Jordan derivation on $\Lambda$ is a
derivation.
\end{theorem}

\begin{proof}
If the quiver only contains one vertex, then $\Lambda\simeq
{\mathcal K}$. Since ${\mathcal K}$ is a field, every Jordan
derivation on ${\mathcal K}$ is a derivation.

Now suppose that the number of vertices in $\Gamma$ is not less than
2. Note that $\Gamma$ is a quiver without oriented cycles. Take a
source $i$ in $\Gamma$ and let $\overline e_i$ be the corresponding
idempotent in $\Lambda$. Then we have
$$
\Lambda\simeq \left[
\begin{array}
[c]{cc}%
(1-{\overline e_i})\Lambda(1-{\overline e_i}) & (1-{\overline e_i})\Lambda {\overline e_i}\\
0 & {\mathcal K}\\
\end{array}
\right].
$$

Let us denote $(1-{\overline e_i})\Lambda(1-{\overline e_i})$ by
$\Lambda'$. We assert that each Jordan derivation on $\Lambda$ is a
derivation if and only if each Jordan derivation on $\Lambda'$ is a
derivation. In fact, the assertion holds true when $(1-{\overline
e_i})\Lambda {\overline e_i}=0$. If $(1-{\overline e_i})\Lambda
{\overline e_i}\neq 0$, then the assertion follows Lemma
\ref{xxsec3.1} and Lemma \ref{xxsec3.3}. Then it is sufficient to
determine whether every Jordan derivation on $\Lambda'$ is a
derivation. Note that $\Lambda'\simeq {\mathcal K}(\Gamma', \rho')$,
where $(\Gamma', \rho')$ is obtained by removing the vertex $i$ from
$\Gamma$. We continuously repeat this process and ultimately arrive
at the algebra ${\mathcal K}$ after finite times, since $\Gamma_0$
is a finite set. Clearly, every Jordan derivation on ${\mathcal K}$
is a derivation. This completes the proof.
\end{proof}

\begin{corollary}\label{xxsec3.5}
Let ${\mathcal K}$ be a field with ${\rm Char} {\mathcal K} \neq 2$
and $\Lambda={\mathcal K}(\Gamma, \rho)$ a finite dimensional path
algebra of the quiver $(\Gamma, \rho)$ with relations. Suppose that
$\Theta$ is a derivation of $\Lambda$ with $\Theta(x)\in \mathcal
{Z}(\Lambda)$ for all $x\in \Lambda$. Then $\Theta=0$.
\end{corollary}

\begin{proof}
Let $i$ be a source in $\Gamma$. Then $\Lambda={\mathcal K}(\Gamma,
\rho)\simeq \left[
\begin{array}
[c]{cc}%
\Lambda' & (1-{\overline e_i})\Lambda {\overline e_i}\\
0 & {\mathcal K}\\
\end{array}
\right].$ Let $\Theta$ be a derivation of $\Lambda$. By Lemma
\ref{xxsec3.3} and Theorem \ref{xxsec3.4} it follows that $\Theta$
has the form
$$
\begin{aligned}
& \Theta\left(\left[
\begin{array}
[c]{cc}%
a & m\\
0 & k\\
\end{array}
\right]\right) = \left[
\begin{array}
[c]{cc}%
\delta_1(a) & am_0-m_0k+\tau_2(m)\\
0 & 0\\
\end{array}
\right] ,\quad \forall \left[
\begin{array}
[c]{cc}%
a & m\\
0 & k\\
\end{array}
\right]\in \Lambda,
\end{aligned}
$$
where $m_0\in (1-{\overline e_i})\Lambda {\overline e_i}$ and
$\delta_1: \Lambda^\prime\longrightarrow \Lambda^\prime, \tau_2:
(1-{\overline e_i})\Lambda {\overline e_i} \longrightarrow
(1-{\overline e_i})\Lambda {\overline e_i}$ are all ${\mathcal
K}$-linear mappings satisfying conditions
\begin{enumerate}
\item [(1)] $\delta_1$ is a derivation of $\Lambda^\prime$.

\item [(2)] $\tau_2(am)=a\tau_2(m)+\delta_1(a)m.$
\end{enumerate}

Assume that $\Theta(x)\in \mathcal {Z}(\Lambda)$ for all $x\in
\Lambda$. Then
$$
\begin{aligned}
& \Theta\left(\left[
\begin{array}
[c]{cc}%
a & m\\
0 & b\\
\end{array}
\right]\right) = \left[
\begin{array}
[c]{cc}%
\delta_1(a) & 0\\
0 & 0\\
\end{array}
\right] ,\quad \forall \left[
\begin{array}
[c]{cc}%
a & m\\
0 & b\\
\end{array}
\right]\in \Lambda,
\end{aligned}
$$
where $\delta_1(x)\in \mathcal {Z}(\Lambda')$ for all
$x\in\Lambda'$. This implies that $\Theta\neq 0$ if and only if
$\delta_1\neq 0$. Since $\Gamma_0$ is a finite set, repeat this
process finite times we obtain that $\Theta\neq 0$ if and only if
some derivation on $K$ is nonzero. However, every derivation on $K$
is zero. This forces $\Theta$ to be zero.
\end{proof}

Moreover, the proof of Theorem \ref{xxsec3.4} implies that one-point
extension preserves the property of every Jordan derivation being a
derivation.

\begin{corollary}
Let $\mathcal {K}$ be a field and $\mathcal {A}$ a finite
dimensional $\mathcal {K}$-algebra with every Jordan derivation
being a derivation. Let $\Lambda$ be a one-point extension of
$\mathcal {A}$ by the bimodule $_\mathcal {A}\mathcal {M}_\mathcal
{K}$. Then every Jordan derivation of $\Lambda$ is a derivation.
\end{corollary}

At the end of this section, let us characterize Jordan generalized
derivations and generalized Jordan derivations of $\Lambda$. Recall
that a linear map $f:\Lambda\rightarrow \Lambda$ is called a
\emph{Jordan generalized derivation} if there exists a linear map
$d:\Lambda\rightarrow\Lambda$ such that
$$f(x\circ y)=f(x)\circ y+x\circ d(y)\eqno(3.1)$$
for all $x,y\in\Lambda$,where $d$ is called an associated linear map
of $f$. A linear map $f :\Lambda\longrightarrow \Lambda$ is called a
\emph{generalized Jordan derivation} if there exists a linear map $d
:\Lambda\longrightarrow \Lambda$ such that $f(x\circ
y)=f(x)y+f(y)x+xd(y)+yd(x)$ for all $x, y\in \Lambda$. A linear map
$f :\Lambda\longrightarrow \Lambda$ is called a \emph{generalized
derivation} if there exists a linear map $d :\Lambda\longrightarrow
\Lambda$ such that $f(xy)=f(x)y+xd(y)$ for all $x, y\in \Lambda$.

\begin{proposition}
Let ${\mathcal K}$ be a field with ${\rm Char} {\mathcal K} \neq 2$
and $\Lambda={\mathcal K}(\Gamma, \rho)$ a finite dimensional path
algebra of the quiver $(\Gamma, \rho)$ with relations.  Then
\begin{enumerate}
\item [(1)] Every Jordan generalized derivation of $\Lambda$ is a
generalized derivation.
\item [(2)] Every generalized Jordan derivation of $\Lambda$ is a
generalized derivation.
\end{enumerate}
\end{proposition}

\begin{proof}
(1) Let $f$ be a Jordan generalized derivation on $\Lambda$.
Firstly, we claim that $f(1)\in \mathcal {Z}(\Lambda)$. In fact,
since $\Gamma$ has no oriented cycles, let us assume that $i$ is a
source in $\Gamma_0$. Then
$$\Lambda={\mathcal K}(\Gamma, \rho)\simeq \left[
\begin{array}
[c]{cc}%
\Lambda' & (1-{\overline e_i})\Lambda {\overline e_i}\\
0 & {\mathcal K}\\
\end{array}
\right].$$ It
follows from \cite{LiBenkovic} Lemma 2.4  that
$$f(1)=\left[\begin{array}[c]{cc}%
(1-{\overline e_i})f(1)(1-{\overline e_i}) & 0\\
0 & {\overline e_i}f(1){\overline e_i}\\ \end{array} \right]$$ and
$[[x,y], f(1)]=0$ for all $x, y\in \Lambda$. Note that $K$ is a
field and then clearly ${\overline e_i}f(1){\overline e_i}\in
\mathcal {Z}(\mathcal {K})=\mathcal {K}$. On the other hand, taking
$x=\left[\smallmatrix {\overline e_i} & 0\\
0 & 0 \endsmallmatrix \right]$ and $y=\left[\smallmatrix 0 & m\\
0 & 0 \endsmallmatrix \right]$, where $m\in (1-{\overline
e_i})\Lambda {\overline e_i}$, in $[[x,y], f(1)]=0$ leads to
$$(1-{\overline e_i})f(1)(1-{\overline e_i})m=m{\overline
e_i}f(1){\overline e_i}.$$ This implies that $f(1)\in \mathcal
{Z}(\Lambda)$ if and only if $(1-{\overline e_i})f(1)(1-{\overline
e_i})\in \mathcal {Z}(\Lambda')$. Repeat this process finite times,
we arrive at the algebra $\mathcal {K}$. Then the commutativity of
$\mathcal {K}$ gives $f(1)\in \mathcal {Z}(\Lambda)$. Hence we have
from \cite{LiBenkovic} Theorem 2.3 that for all $x\in \Lambda$,
$f(x)=f(1)x+d(x)$, where $d$ is a Jordan derivation of $\Lambda$. By
Theorem \ref{xxsec3.4}, $d$ is a derivation of $\Lambda$. It follows
from \cite{LiBenkovic} Proposition 2.1 that $f$ is a generalized
derivation of $\Lambda$.

(2) Let $f$ be a generalized Jordan derivation of $\Lambda$. Then we
have from \cite{Benkovic4} Lemma 4.1 that for all $x\in \Lambda$,
$f(x)=f(1)x+d(x)$, where $d$ is a Jordan derivation of $\Lambda$. By
Then $d$ is a derivation of $\Lambda$ Theorem \ref{xxsec3.4}. It
follows from \cite{LiBenkovic} Proposition 2.1 that $f$ is a
generalized derivation of $\Lambda$.
\end{proof}

\begin{remark}
Lemma 4.1 of \cite{Benkovic4} and Lemma 2.4 of \cite{LiBenkovic}
were obtained without the faithful assumption of bimodule $_\mathcal
{A}\mathcal {M}_\mathcal {B}$.
\end{remark}

\bigskip

\section{Lie derivations on path algebras}\label{xxsec4}

In \cite{Cheung}, Cheung characterized Lie derivations of triangular
algebras. The general form of Lie derivations was described and a
sufficient and necessary condition which enables every Lie
derivation to be standard was given. The current authors
\cite{LiWei1} studied the structure of Lie derivations of
generalized matrix algebras and also provided certain sufficient
condition such that every Lie derivation being of the standard form
$(\spadesuit)$. In this section we will investigate Lie derivations
on a class of path algebras without the faithful assumption.

\begin{lemma}\cite[Proposition 4]{Cheung}\label{xxsec4.1}
A Lie derivation $\Theta$ from $\mathcal{T_K}=\left[\smallmatrix {\mathcal A} & {\mathcal M}\\
0 & {\mathcal B} \endsmallmatrix \right]$ into itself is of the form
$$
\Theta\left(\left[
\begin{array}
[c]{cc}%
a & m\\
0 & b\\
\end{array}
\right]\right) =\left[
\begin{array}
[c]{cc}%
\delta_1(a)+\delta_4(b) & am_0-m_0b+\tau_2(m)\\
0 & \mu_1(a)+\mu_4(b)\\
\end{array}
\right] , \forall \left[
\begin{array}
[c]{cc}%
a & m\\
0 & b\\
\end{array}
\right]\in \mathcal{T_K},
$$
where $m_0\in {\mathcal M}$ and
$$
\begin{aligned} \delta_1:& {\mathcal A} \longrightarrow {\mathcal A}, & \delta_4:& {\mathcal B} \longrightarrow {\mathcal Z(A)}, &
 \tau_2: & {\mathcal M}\longrightarrow {\mathcal M}, &
\mu_1: & {\mathcal A}\longrightarrow {\mathcal Z(B)} & \mu_4: &
{\mathcal B}\longrightarrow {\mathcal B}
\end{aligned}
$$
are all $\mathcal{K}$-linear mappings satisfying the following
conditions
\begin{enumerate}
\item[(1)] $\delta_1$ is a Lie derivation of ${\mathcal A}$, $\mu_1([a, a'])=0$
for all $a, a'\in \mathcal{A}$ and
$\tau_2(am)=a\tau_{2}(m)+\delta_1(a)m-m\mu_1(a)$ for all $a\in
\mathcal{A}, m\in \mathcal{M};$

\item[(2)] $\mu_4$ is a Lie derivation of ${\mathcal B}$, $\delta_4([b,
b'])=0$ for all $b, b'\in \mathcal{B}$ and
$\tau_2(mb)=\tau_2(m)b+m\mu_4(b)-\delta_4(b)m$ for all $b\in
\mathcal{B}, m\in \mathcal{M}.$
\end{enumerate}

Furthermore, a Lie derivation $\Theta$ on $\mathcal{T_K}$ is of the
standard form $(\spadesuit)$ if and only if there exist linear
mappings $l_{\mathcal A}: {\mathcal A}\rightarrow {\mathcal Z(A)}$
and $l_{\mathcal B}: {\mathcal B}\rightarrow {\mathcal Z(B)}$
satisfying
\begin{enumerate}
\item [(3)] $\delta=\delta_1-l_{\mathcal A}$ is a derivation on $\mathcal{A}$, $l_{\mathcal A}([a,
a'])=0$ for all $a, a'\in \mathcal{A}$ and $l_{\mathcal
A}(a)m=m\mu_1(a)$ for all $a\in \mathcal{A}, m\in \mathcal{M};$

\item[(4)] $\mu=\mu_1-l_{\mathcal B}$ is a derivation on $\mathcal{B}$, $l_{\mathcal B}([b,
b'])=0$ for all $b, b'\in \mathcal{B}$ and $ml_{\mathcal
B}(b)=\delta_4(b)m$ for all $b\in \mathcal{B}, m\in \mathcal{M}.$
\end{enumerate}
\end{lemma}

Let $\mathcal {R}$ be a commutative ring with identity and
$\mathcal{A}$ be an $\mathcal{R}$-algebra. Then $\mathcal{W(A)}$
defined in \cite{Cheung} is the smallest subalgebra of $\mathcal{A}$
satisfying the following conditions:

\begin{enumerate}
\item[(1)] $[x,y]\in \mathcal{W(A)}$ for all $x, y\in \mathcal{A};$

\item[(2)] Suppose that $x\in \mathcal{A}$ and $f(t)$ is a polynomial
in $\mathcal{R}[t]$. If $f'(x)=0$, then $f(x)\in \mathcal{W(A)};$

\item[(3)] Suppose that $x\in \mathcal{A}$ and $f(t)$ is a polynomial
in $\mathcal{R}[t]$. If $f(x)\in \mathcal{W(A)}$ and $f'(x)$ is
invertible, then $x\in \mathcal{W(A)};$

\item[(4)] $\mathcal{W(A)}$ contains all the idempotents in $\mathcal{A};$

\item[(5)] $\mathcal{W(A)}$ contains all the elements of the form $x^p$,
where $x\in \mathcal{A}$ and $p\geq 0$ is the characteristic of
$\mathcal{A};$

\item[(6)] $\{x\in \mathcal{W(A)}\mid x^{-1}\in \mathcal{W(A)}\}\subseteq
\mathcal{W(A)};$

\item[(7)] If $x\in \mathcal{A}$ is invertible with $x^n\in \mathcal{W(A)}$
for some positive integer $n$, then $nx\in \mathcal{W(A)}.$
\end{enumerate}

Then Cheung in \cite{Cheung} gives a sufficient condition for a Lie
derivation of $\mathcal{T_K}$ being of the standard form.

\begin{lemma}\cite[Theorem 11]{Cheung}\label{xxsec4.3}
Let $\mathcal{T_K}=\left[\smallmatrix {\mathcal A} & {\mathcal M}\\
0 & {\mathcal B}\endsmallmatrix \right]$ be a triangular algebra.
Then every Lie derivation of $\mathcal{T_K}$ is of the standard form
$(\spadesuit)$ if the following conditions are satisfied:
\begin{enumerate}
\item[(1)] Every Lie derivation of $\mathcal{A}$ is of standard form and ${\mathcal A}={\mathcal W}\mathcal{(A)};$
\item[(2)] Every Lie derivation of ${\mathcal B}$ is of standard form and
${\mathcal B}={\mathcal W}\mathcal{(B)}.$
\end{enumerate}
\end{lemma}

Since a field $\mathcal{K}$ is an algebra over itself, obviously
${\mathcal K}=\mathcal{W(K)}$. Surprisingly, for a finite
dimensional path algebra $\Lambda=\mathcal {K}(\Gamma, \rho)$, the
equality $\mathcal {W}(\Lambda)=\Lambda$ also holds.

\begin{lemma}\label{xxsec4.2}
Let $\mathcal{K}$ be a field and $\Gamma$ a finite quiver without
oriented cycles. Let $\Lambda={\mathcal K}(\Gamma, \rho)$ be the
path algebra of $\Gamma$ with relations. Then ${\mathcal
W}(\Lambda)=\Lambda$.
\end{lemma}

\begin{proof}
If $\Gamma$ only contains one vertex, then $\Lambda\simeq {\mathcal
K}$. In this case we get ${\mathcal W}(\Lambda)=\Lambda$.

Now suppose that the number of vertices in $\Gamma$ is not less than
2. It follows from the condition (4) in the definition of ${\mathcal
W}(\Lambda)$ that all the trivial paths are contained in ${\mathcal
W}(\Lambda)$. On the other hand, for an arbitrary nontrivial path
$\overline p=\overline\alpha_n\cdots\overline\alpha_1$, we have
$\overline p=[\overline p, s(\overline p)]$, which is due to the
fact $\overline ps(\overline p)=\overline p$ and $s(\overline
p)\overline p=0$. Then the condition (1) of the definition of
${\mathcal W}(\Lambda)$ implies that $\overline p\in {\mathcal
W}(\Lambda)$. Therefore all paths in $(\Gamma, \rho)$ are contained
in ${\mathcal W}(\Lambda)$. Hence $\Lambda={\mathcal W}(\Lambda)$.
\end{proof}

Now we are ready to give the main result of this section.

\begin{theorem}\label{xxsec4.4}
Let $\Lambda={\mathcal K}(\Gamma, \rho)$ a finite dimensional path
algebra of the quiver $(\Gamma, \rho)$ with relations and $\Theta$
be a Lie derivation from $\Lambda$ into itself. Then $\Theta$ is of
the standard form $(\spadesuit)$. Moreover, the standard
decomposition is unique.
\end{theorem}

\begin{proof}
If $\Gamma$ only contains one vertex, then $\Lambda\simeq {\mathcal
K}$. Clearly, every Lie derivation on $\mathcal{K}$ is of the
standard form $(\spadesuit)$.

Now assume that the number of vertices in $\Gamma$ is not less than
2. Take a source $i$ in $\Gamma$ and let $\overline e_i$ be the
corresponding idempotent in $\Lambda$. Then
$$
\Lambda\simeq \left[
\begin{array}
[c]{cc}%
(1-{\overline e_i})\Lambda(1-{\overline e_i}) & (1-{\overline e_i})\Lambda {\overline e_i}\\
0 & {\mathcal K}\\
\end{array}
\right].$$

Note that each Lie derivation on ${\mathcal K}$ is of the standard
form $(\spadesuit)$ and ${\mathcal W}({\mathcal K})={\mathcal K}$.
Let us denote $(1-{\overline e_i})\Lambda(1-{\overline e_i})$ by
$\Lambda'$. In view of Lemma \ref{xxsec4.2} we have
$\Lambda'={\mathcal W}(\Lambda')$. By Lemma \ref{xxsec4.3}, we
obtain that each derivation on $\Lambda$ is of the standard form if
and only if each Lie derivation on $\Lambda'$ is of the standard
from. Repeating this process continuously, we ultimately conclude
that each Lie derivation on $\Lambda$ is of the standard form if and
only if each Lie derivation on ${\mathcal K}$ is of the standard
form. This is due to the fact that $\Gamma$ is a finite quiver.
Therefore every Lie derivation on $\Lambda$ is of the standard form
$(\spadesuit)$.

Let us see the uniqueness of the standard decomposition. Suppose
that $\Theta=D+\Phi=D'+\Phi'$. Then $D-D'=\Phi'-\Phi$. Clearly, the
image of $\Phi'-\Phi$ is in $\mathcal {Z}(\Lambda)$. This shows that
$(D-D')(x)\in \mathcal {Z}(\Lambda)$ for all $x\in \Lambda$. It
follows from Corollary \ref{xxsec3.5} that $D-D'=0$. Consequently,
$D=D'$ and $\Phi=\Phi'$.
\end{proof}

\begin{corollary}\label{xxsec4.5}
Let $\Theta_{\rm Lied}$ be a Lie derivation of $\Lambda={\mathcal
K}(\Gamma, \rho)$. Then there exists a derivation $D$ of $\Lambda$
such that $\Theta_{\rm Lied}(x)=D(x)$ for all $x=\sum_ik_i\overline
p_i\in \Lambda$, where $p_i$ are non-trivial paths in $\Gamma$.
\end{corollary}

\begin{proof} Let us denote by $\Gamma'$ the
quiver obtained via adding a vertex to $\Gamma$ with no arrows
starting and ending at it. Then
$$\Lambda'=K(\Gamma', \rho)\simeq \left[
\begin{array}
[c]{cc}%
\Lambda & 0\\
0 & {\mathcal K}\\
\end{array}
\right].
$$
Let $\Theta_{\rm Lied}$ be a Lie derivation on $\Lambda$ and let $f$
be a Lie derivation on $\mathcal{K}$. It follows from Lemma
\ref{xxsec4.1} that
$$
\begin{aligned}
& \Theta'_{\rm Lied}\left(\left[
\begin{array}
[c]{cc}%
x & 0\\
0 & b\\
\end{array}
\right]\right) =& \left[
\begin{array}
[c]{cc}%
\Theta_{\rm Lied}(x) & 0\\
0 & f(b)\\
\end{array}
\right] ,  & \quad \forall \left[
\begin{array}
[c]{cc}%
x & 0\\
0 & b\\
\end{array}
\right]\in \Lambda'
\end{aligned}
$$
is a Lie derivation on $\Lambda'$. We have from Theorem
\ref{xxsec4.4} that $\Theta'_{\rm Lied}$ is of standard form. Then
by Lemma \ref{xxsec4.1}, there exists linear map $\phi_{\rm
\Lambda}:\Lambda\rightarrow Z(\Lambda)$ such that $\Theta_{\rm
Lied}-\phi_{\rm \Lambda}$ is a derivation and $\phi_{\rm\Lambda}([x,
y])=0$ for all $x, y\in \Lambda$. For a nontrivial path $p\in
\Gamma$, it is easy to check that $\overline p=[\overline p, s(p)]$.
Thus $\phi_{\rm \Lambda}(\overline p)=0$. This implies that for
arbitrary $x=\sum_ik_i\overline p_i\in \Lambda$, where $p_i$ is a
non-trivial path in $\Gamma$, $\phi_{\rm \Lambda}(x)=0$. Define $D$
to be $\Theta_{\rm Lied}-\phi_{\rm \Lambda}$. Then $D(x)=\Theta_{\rm
Lied}(x)-\phi_{\rm \Lambda}(x)=\Theta_{\rm Lied}(x)$.
\end{proof}

\begin{lemma}\cite[Theorem 2.5]{GuoLi}\label{xxsec4.6}
A linear mapping $\Theta$ from a path algebra ${\mathcal K}\Gamma$
into itself is a derivation if and only if $\Theta$ satisfies
\begin{enumerate}
\item[(1)] $\Theta(e_i)=\sum\limits_{q\in \mathscr{P}, s(q)\neq e(q)\atop s(q)=i\,\, or \,\, e(q)=i}k_q^{e_i}q$;
\item[(2)] For each nontrivial path $p$,
$$\Theta(p)=\sum_{q\in \mathscr{P}, s(q)\neq e(q)\atop e(q)=s(p)}k_q^{s(p)}qp+\sum_{q\in \mathscr{P}, s(q)=s(p)\atop e(q)=e(p)}k_q^pq
+\sum_{q\in \mathscr{P}, s(q)\neq e(q)\atop
s(q)=e(p)}k_q^{e(p)}pq.$$
\end{enumerate}
The coefficients $c_q^p$ are subject to certain
conditions.\footnote{See \cite{GuoLi} Theorem 2.5 for details.}
\end{lemma}

\begin{lemma}\label{xxsec4.7}
Let $\Gamma$ be a connected quiver without oriented cycles and
${\mathcal K}\Gamma$ be the path algebra of $\Gamma$. If $\Theta$ is
a Lie derivation of ${\mathcal K}\Gamma$ with the standard
decomposition $\Theta=D+\Phi$, then  we have
$$
\Phi(e_i)=k_i
$$
for an arbitrary trivial path $e_i\in \Gamma$, where $k_i\in
{\mathcal K}$.
\end{lemma}

\begin{proof}
Suppose that
$$
\Phi(e_i)=\sum_{j}k_je_j+\sum_{p\in \mathscr{P}, s(p)\neq e(p)}
k_pp.\eqno(4.1)
$$
Clearly, for every trivial path $e_t$,
$$
e_t\Phi(e_i)=k_te_t+\sum_{p\in \mathscr{P}, s(p)\neq e(p), e(p)=t
}k_pp.\eqno(4.2)
$$ On the
other hand,
$$
\Phi(e_i)e_t=k_te_t+\sum_{p\in \mathscr{P}, s(p)\neq e(p),
s(p)=t}k_pp.\eqno(4.3)
$$
Note that $\Phi(x)\in \mathcal {Z}(\Lambda)$ for all $x\in \Lambda$.
Combining (4.2) with (4.3) leads to
$$
\sum_{p\in\mathscr{P}, s(p)\neq e(p), e(p)=t}k_pp=\sum_{p\in
\mathscr{P}, s(p)\neq e(p), s(p)=t}k_pp.
$$
This implies
that for all nontrivial path $p$ with $s(p)=t$ or $e(p)=t$, the
coefficients $k_p=0$. Note that $e_{t}$ is arbitrary. Thus the
coefficients of all nontrivial paths is zero. So (4.1) becomes to
$$
\Phi(e_i)=\sum_{j}k_je_j.
$$

Now assume that $\alpha$ is an arrow in $\Gamma$ with $s(\alpha)=x$
and $e(\alpha)=y$. It is easy to check that
$\alpha\Phi(e_i)=k_x\alpha$ and $\Phi(e_i)\alpha=k_y\alpha$. It
follows from $\Phi(\Lambda)\in \mathcal {Z}(\Lambda)$ that
$k_x=k_y$. Note that $\Gamma$ is a connected quiver. Then the
coefficients of all trivial paths in $\Phi(e_i)$ are the same. That
is, $\Phi(e_i)=k_i\sum_je_j=k_i$.
\end{proof}

We can characterize Lie derivations of a finite dimensional path
algebra now.

\begin{theorem}
Let $\Gamma$ a quiver without oriented cycles and ${\mathcal
K}\Gamma$ be a finite dimensional path algebra of $\Gamma$. Then a
linear mapping $\Theta$ from ${\mathcal K}\Gamma$ into itself is a
Lie derivation if and only if $\Theta$ satisfies the following
conditions
\begin{enumerate}
\item[(1)] $\Theta(e_i)=k_i+\sum\limits_{q\in \mathscr{P}, s(q)\neq e(q)\atop s(q)=i\,\, or\,\, e(q)=i}k_q^{e_i}q$;
\item[(2)] For each nontrivial path $p$,
$$\Theta(p)=\sum_{q\in \mathscr{P}, s(q)\neq e(q)\atop e(q)=s(p)}k_q^{s(p)}qp+\sum_{q\in \mathscr{P}, s(q)=s(p)\atop e(q)=e(p)}k_q^pq
+\sum_{q\in \mathscr{P}, s(q)\neq e(q)\atop
s(q)=e(p)}k_q^{e(p)}pq.$$
\end{enumerate}
The coefficients $k_q^p$ here are subject to certain conditions the
same as in \cite{GuoLi} Theorem 2.5.
\end{theorem}

\begin{proof}
It follows from Theorem \ref{xxsec4.4}, Lemma \ref{xxsec4.6} and
Lemma \ref{xxsec4.7} easily.
\end{proof}

\begin{example}
Let $\Gamma$ be a quiver as follows.
$$\xymatrix@C=13mm{
  \bullet
  \ar@<0pt>[r]_(0){1}^{\alpha}  &
  \bullet &
  \bullet
  \ar@<0pt>@[r][l]_(0.5){\beta}^(0){3}^(1){2} }$$

Let $\Theta$ be a linear mapping from ${\mathcal K}\Gamma$ into
itself defined by $\Theta(e_1)=k_1-\alpha$,
$\Theta(e_2)=k_2+\alpha+\beta$, $\Theta(e_3)=k_3-\beta$ and
$\Theta(\alpha)=\Theta(\beta)=0$, where $k_i\in {\mathcal K}$ for
$i=1, 2, 3$. If there exists some $i\in \{1, 2, 3\}$ such that
$k_i\neq 0$, then $\Theta$ is a Lie derivation of ${\mathcal
K}\Gamma$ but not a derivation.
\end{example}

\bigskip


\end{document}